\documentclass{amsart}

\usepackage[hyperindex]{hyperref}

\usepackage{amsthm}
\usepackage{amssymb}
\usepackage{amsrefs}
\usepackage{bbold}
\usepackage{mathrsfs}

\numberwithin{equation}{section} 
\newtheorem{thm}{Theorem}[section]
  \newtheorem{lem}[thm]{Lemma}
\newtheorem{lemma}[thm]{Lemma}
  \newtheorem{prop}[thm]{Proposition}
\theoremstyle{definition}

  \theoremstyle{remark}
  
\newtheorem{remark}[thm]{Remark}
\IfFileExists{d-bends.sty}{\usepackage{d-bends}}
{\def\dangbend{ALERT}
\def\dangbendpar{\hbox to 20pt{\hss ALERT\enspace}}
\def\db{\hangindent=20pt\hangafter=-2\noindent
\hbox to 0pt{\hss\hbox to
20 pt{\dangbendpar\hss}}}
\def\altdb{\vadjust{\vbox to 0pt{\vss\hbox{\kern \hsize
\quad{\dangbend}}\kern\baselineskip\kern-10pt}}}}

\DeclareMathOperator{\Ind}{Ind}

\DeclareMathOperator{\Prim}{Prim}
\DeclareMathOperator*{\supp}{supp}
\DeclareMathOperator{\Aut}{Aut}
\DeclareMathOperator{\ind}{ind}

\newcommand{\go}{G^{(0)}}
\newcommand{\so}{\Sigma^{(0)}}
\newcommand{\indtog}[1]{\Ind_{#1}^{G}}
\newcommand{\indgug}{\indtog{G(u)}}

\let\tensor=\otimes

\renewcommand\H{\mathcal{H}}
\newcommand\atensor{\odot}

\def\mathcs{C^{*}}
\newcommand{\cs}{\ensuremath{\mathcs}}

\newcommand{\dint}{\int^{\oplus}}

\newcommand{\primcss}{\Prim\cs(\Sigma)}
\newcommand{\G}{\underline{\mathcal G}}
\newcommand{\css}{\mathcal S}

\newcommand{\tr}{\tilde r}
\newcommand\hr{\hat r}
\newcommand{\set}[1]{\{\,#1\,\}}

\newcommand{\dirintfont}[1]{\mathscr{#1}}
\newcommand{\HH}{\dirintfont{H}}
\newcommand{\KK}{\dirintfont{K}}
\newcommand\VV{\dirintfont{V}}
\newcommand{\restr}[1]{|_{{#1}}}
\newcommand\tk{\tilde{\mathcal{K}}}
\newcommand\deltau{{\delta'}}
\newcommand\tH{\tilde H}
\newcommand\tHH{\tilde \HH}

\newcommand\Rind{\underline{L}}

\let\phi\varphi
\def\<{\langle}
\def\>{\rangle}
\let\ipscriptstyle=\scriptscriptstyle
\def\lipsqueeze{{\mskip -3.0mu}}
\def\ripsqueeze{{\mskip -3.0mu}}
\def\ipcomma{\nobreak\mathrel{,}\nobreak}
\newbox\ipstrutbox
\setbox\ipstrutbox=\hbox{\vrule height8.5pt
width 0pt}
\def\ipstrut{\copy\ipstrutbox}
\def\lip#1<#2,#3>{\mathopen{\relax_{\ipstrut\ipscriptstyle{
#1}}\lipsqueeze
\langle} #2\ipcomma #3 \rangle}
\def\blip#1<#2,#3>{\mathopen{\relax_{\ipstrut
\ipscriptstyle{ #1}}\lipsqueeze\bigl\langle} #2\ipcomma #3 \bigr\rangle}
\def\rip#1<#2,#3>{\langle #2\ipcomma #3
\rangle_{\ripsqueeze\ipstrut\ipscriptstyle{#1}}}
\def\brip#1<#2,#3>{\bigl\langle #2\ipcomma #3
\bigr\rangle_{\ripsqueeze\ipstrut\ipscriptstyle{#1}}}
\def\angsqueeze{\mskip -6mu}
\def\smangsqueeze{\mskip -3.7mu}
\def\trip#1<#2,#3>{\langle\smangsqueeze\langle #2\ipcomma #3
\rangle\smangsqueeze\rangle_{\ripsqueeze\ipstrut\ipscriptstyle{#1}}}
\def\btrip#1<#2,#3>{\bigl\langle\angsqueeze\bigl\langle #2\ipcomma
#3
\bigr\rangle
\angsqueeze\bigr\rangle_{\ripsqueeze\ipstrut\ipscriptstyle{#1}}}
\def\tlip#1<#2,#3>{\mathopen{\relax_{\ipstrut\ipscriptstyle{
#1}}\lipsqueeze \langle\smangsqueeze\langle} #2\ipcomma #3
\rangle\smangsqueeze\rangle}
\def\btlip#1<#2,#3>{\mathopen{\relax_{\ipstrut\ipscriptstyle{
#1}}\lipsqueeze
\bigl\langle\angsqueeze\bigl\langle} #2\ipcomma #3
\bigr\rangle\angsqueeze\bigr\rangle}

\def\ip(#1|#2){(#1\mid #2)}
\def\bip(#1|#2){\bigl(#1 \mid #2\bigr)}
\def\Bip(#1|#2){\Bigl( #1 \bigm| #2 \Bigr)}

\newcommand\Rip{\rip \scriptstyle\star}

\renewcommand\MR[1]{\relax}

\newcommand\Lpp{L''}
\newcommand\Lp{L'}
\newcommand\lambdabar{\underline{\lambda}}
\def\charfcn#1{\mathbb{1}_{#1}}

\newcommand\bb{\mathcal{B}^{b}}
\newcommand\cb{C^{b}}
\newcommand\half{\frac12}

\begin{document}

\title[The Effros-Hahn Conjecture]{The Generalized Effros-Hahn
  Conjecture for Groupoids}

\author{Marius Ionescu}
\address{Department of Mathematics \\ Cornell University \\ Ithaca, NY
14853-4201}

\curraddr{Department of Mathematics \\ University of Connecticut,
  Storrs, CT 06269-3009}

\email{ionescu@math.uconn.edu}

\author{Dana Williams}
\address{Department of Mathematics \\ Dartmouth College \\ Hanover, NH
03755-3551}

\email{dana.p.williams@Dartmouth.edu}

\begin{abstract}
  The generalized Effros-Hahn conjecture for groupoid \cs-algebras
  says that, if $G$ is amenable, then every primitive ideal of the
  groupoid \cs-algebra $\cs(G)$ is induced from a stability group.  We
  prove that the conjecture is valid for all second countable amenable
  locally compact Hausdorff groupoids.  Our results are a sharpening
  of previous work of Jean Renault and depend significantly on his
  results.
\end{abstract}

\date{28 July 2008}


\maketitle

\section{Introduction}
\label{sec:introduction}

A dynamical system $(A,G,\alpha)$, where $A$ is a \cs-algebra, $G$ is
a locally compact group and $\alpha$ is a strongly continuous
homomorphism of $G$ into $\Aut A$, is called \emph{EH-regular} if
every primitive ideal of the crossed product $A\rtimes_{\alpha}G$ is
induced from a stability group (see
\cite{wil:crossed}*{Definition~8.18}).  In their 1967 \emph{Memoir}
\cite{effhah:mams67}, Effros and Hahn conjectured that if $(G,X)$ was
a second countable locally compact transformation group with $G$
amenable, then $\bigl(C_{0}(X),G,\operatorname{lt}\bigr)$ should be
EH-regular.  This conjecture, and its generalization to dynamical
systems, was proved by Gootman and Rosenberg in \cite{gooros:im79}
building on results due to Sauvageot
\citelist{\cite{sau:jfa79}\cite{sau:ma77}}. For additional comments on
this result, its applications, as well as precise statement and proof,
see \cite{wil:crossed}*{\S8.2 and Chap.~9}.

In \cite{ren:jot91}, Renault gives the following version of the 
Gootman-Rosenberg-Sauvageot Theorem for groupoid dynamical systems.
Let $G$ be a locally compact groupoid and $(A,G,\alpha)$
a groupoid dynamical system.  If $R$ is a representation of the
crossed product $A\rtimes_{\alpha}G$, then Renault forms the
restriction, $\hat L$, of $R$ to the isotropy groups of $G$ and forms
an induced representation $\Ind \hat L$ of $A\rtimes_{\alpha} G$ such
that $\ker R\subset \ker (\Ind \hat L)$
\cite{ren:jot91}*{Theorem~3.3}.  When $G$ is suitably amenable, then
the reverse conclusion holds \cite{ren:jot91}*{Theorem~3.6}.  This is
a powerful result and allows Renault to establish some very striking
results concerning the ideal structure of crossed products and has deep
applications to the question of when a crossed product is simple (see
\cite{ren:jot91}*{\S4}).

In this note, our object is to provide a significant sharpening of
Renault's result in the case a groupoid \cs-algebra --- that is, a
dynamical system where 
$G$ acts on the commutative algebra
$C_{0}(\go)$ by translation.  We aim to show that if $G$ is Hausdorff
and amenable, then every primitive
ideal $K$ of $\cs(G)$ is induced from a stability group.  That is, we
show that $K=\Ind_{G(u)}^{G}J$ for a primitive ideal $J$ of
$\cs\bigl(G(u)\bigr)$, where $G(u)$ is the stability group at some
$u\in\go$.  This not only provides a cleaner generalization of
the Gootman-Rosenberg-Sauvageot result to the groupoid setting, but
gives us a much better means to study the fine ideal structure of
groupoids and the primitive ideal space (together with its topology)
in particular.  (For further discussion of these ideas, see
\cite{echwil:tams08}*{\S4}.)

In Section~\ref{sec:main-result} we give a careful statement of the
main result, and give a brief summary of some of the tools and
ancillary results needed in the sequel.  Since the proof of the main
result
is
rather involved, we also give an overview of the proof to make the
subsequent 
details easier to parse.  Then in
Section~\ref{sec:proof-main-theorem}, we give the proof itself.  Our
techniques require that we work whenever possible with separable
\cs-algebras.  Therefore all our groupoids are assumed to be second
countable.  We also assume that our locally compact groupoids have
continuous Haar systems and are Hausdorff.  We also assume that
homomorphisms between \cs-algebras are $*$-preserving and that
representations are nondegenerate.

\section{The Main Result}
\label{sec:main-result}
Unlike the situation for groups, the definition of amenability of a
locally compact groupoid is a bit controversial.  The currently
accepted definition originally comes from \cite{ren:groupoid}*{p.~92}:
a locally compact groupoid $G$ with continuous Haar system
$\lambda=\set{\lambda^{u}}$ is \emph{amenable} if there is a net
$\set{\phi_{i}}\subset C_{c}(G)$ such that
\begin{enumerate}
\item the functions $u\mapsto
  \int_{G}|\phi_{i}(\gamma)|^{2}\,d\lambda^{u}(\gamma) $ are uniformly
  bounded, and
\item the functions $\gamma\mapsto \phi_{i}*\phi_{i}^{*}(\gamma)$
  converge to the constant function $1$ uniformly on compacta.
\end{enumerate}
If $G$ is a group, then we recover the usual notion of amenability
(for example, see \cite{wil:crossed}*{Proposition~A.17}).  A different
definition of amenability for a locally compact groupoid is given in
\cite{anaren:amenable00}*{Definition~2.2.8}.\footnote{Both the
  numbering and the statements of some results in the published
  version of this paper differ from those in the
  widely circulated preprint.}
Fortunately, \cite{anaren:amenable00}*{Proposition~2.2.13(iv)},
implies the two definitions are equivalent (and gives some additional
equivalent conditions).  In particular,
\cite{anaren:amenable00}*{Theorem~2.2.13} implies that
amenability is preserved under equivalence of groupoids as defined in
\cite{mrw:jot87}*{Definition~2.1}.  Thus the notion of amenability is
independent of the choice of continuous Haar system on $G$.

\begin{thm}
  \label{thm:Main}Assume that $G$ is a second countable locally
  compact Hausdorff groupoid with Haar system $\set
  {\lambda^{u}}_{u\in\go }$.  Assume also that $G$ is amenable.  If
  $K\subset C^{*}(G)$ is a primitive ideal, then $K$ is induced from an
  isotropy group. That is,
  \begin{equation*}
    K=\Ind _{G(u)}^{G}J
  \end{equation*}
  for a primitive ideal $J\in\Prim (\cs (G(u))$.
\end{thm}

\begin{remark}
  \label{rem-weaker}
  Let $R$ be the equivalence relation on $\go$ induced by $G$:
  $r(\gamma)\sim s(\gamma)$ for all $\gamma\in G$.  We give $R$ the
  Borel structure coming from the topology on $R$ induced from $G$
  (which is often finer than the product topology on $R$ viewed as a
  subset of $\go\times\go$).  Since the proof of
  Theorem~\ref{thm:Main} requires only that we are entitled to apply
  Renault's \cite{ren:jot91}*{Theorem~3.6}, Theorem~\ref{thm:Main} is
  valid under the weaker assumption that the Borel equivalence
  relation $R$ is amenable with respect to every quasi-invariant
  measure $\mu$ on $\go$ \cite{ren:jot91}*{Definition~3.4} (see also
  \cite{anaren:amenable00}*{Definition~3.2.8}).  Some other valid
  hypotheses are discussed in \cite{ren:jot91}*{Remark~3.7}.  We have
  decided to use the less technical hypotheses of amenability of $G$
  here, and to leave the more technical, but weaker, hypotheses for
  the interested reader to sort out as needed.
\end{remark}

Even though Theorem~\ref{thm:Main} involves only groupoid
\cs-algebras, our techniques use the theory of groupoid dynamical
systems and their crossed products.  For these, we will employ the
notation and terminology from \cite{muhwil:nyjm08}*{\S4}.  In
particular, our treatment of direct integrals comes from
\cite{muh:cbms} (which was, in turn, motivated by \cite{ram:am76}),
and we suggest \cite{wil:crossed}*{Appendix~F} as a reference.  We
need Renault's disintegration theorem
\cite{ren:jot87}*{Proposition~4.2} for representations $R$ of
$\cs(G)$.  For the statement, notation and basics for this result, we
suggest \cite{muhwil:nyjm08}*{\S7}.  The disintegration result implies
that $R$ is the integrated form of an unitary representation
$(\mu,\go*\HH,V)$ of $G$ consisting of a quasi-invariant measure $\mu$
on $\go$, a Borel Hilbert bundle $\go*\HH$ and a family
$V=\set{V_{\gamma}:\H\bigl(s(\gamma)\bigr)\to\H\bigl(r(\gamma)\bigr)}$
of unitaries so that
\begin{equation*}
\hat V(\gamma)=\bigl(r(\gamma),V_{\gamma},s(\gamma)\bigr).
\end{equation*}
defines a groupoid homomorphism $\hat V:G\to
\operatorname{Iso}(\go*\HH)$.

The proof of Theorem~\ref{thm:Main} occupies the entire next section.
Since the proof is a bit involved, we give a brief overview here in
the hope that it will motivate some of the efforts in the next section.
(The basic outline follows the proof of the
Gootman-Rosenberg-Sauvageot Theorem as proved in
\cite{wil:crossed}*{Chap.~9}.)

We start by fixing $K\in \Prim \cs(G)$ and letting $R$ be an
irreducible representation such that $\ker R=K$.  We assume that $R$
is the integrated form of a unitary representation $(\mu,\go*\HH,V)$.
Since $V$ defines via restriction a representation $r_{u}$ of each
stability group $G(u)=\set{\gamma\in G:r(\gamma)=u=s(\gamma)}$, and
since we can view each $r_{u}$ as a representation of the \cs-algebra
$\cs(\Sigma)$ of the group bundle $\Sigma$ associated to the
collection $\Sigma^{(0)}$ of closed subgroups of $G$, we can form the
direct integral representation
\begin{equation*}
  r:=\dint_{\go}r_{u}\,d\mu(u)
\end{equation*}
of $\cs(\Sigma)$.  We call $r$ the restriction of $R$ to the isotropy
groups of $G$.

A key step is to observe that $(r,V)$ is a covariant representation of
a groupoid dynamical system $\bigl(\cs(\Sigma),G,\alpha\bigr)$ for a
natural action $\alpha$.  Then we can form the representation
$\Lpp=r\rtimes V$ of $\cs(\Sigma)\rtimes_{\alpha}G$.  This allows us
to invoke Renault's impressive \cite{ren:jot91}*{Theorem~2.2} which is
a groupoid equivariant version of Effros's ideal center decomposition
theorem from \cite{eff:tams63} (for more on Effros's result, see
\cite{wil:crossed}*{Appendix~G}).  This allows us to show that $r$ is
equivalent to a representation
\begin{equation*}
  \tr:=\dint_{\primcss}\tr_{P}\,d\nu(P),
\end{equation*}
where each $\tr_{P}$ has kernel $P$, and $\nu$ is a measure on
$\primcss$.  Moreover, $\primcss$ is a right $G$-space for the action
of $G$ induced by $\alpha$, and
\cite{ren:jot91}*{Theorem~2.2} implies that $\nu$ is quasi-invariant
when $\primcss$ is viewed as the unit space of the transformation
groupoid $\G=\primcss*G$.  (Although
$\G$ is not a locally compact groupoid, it is a Borel groupoid with a
Borel Haar system so the definition of quasi-invariant makes perfectly
good sense.)
We then need to work a bit to see that $\nu$ is also ergodic with
respect to the $G$-action on $\primcss$.  

We then define an induced representation $\operatorname{ind}\tr$ of
$\cs(G)$.  As essential component of the proof is
Proposition~\ref{prop-prim-ind} where we use the quasi-invariance and
ergodicity of $\nu$ to show that the kernel of $\operatorname{ind}\tr$
is an induced primitive ideal.  This is a generalization of
Sauvageot's \cite{sau:ma77}*{Lemma~5.4} where the corresponding result
for transformations groups is proved.  Then the final step in our
proof is to observe that $\operatorname{ind}\tr$ is equivalent to the
induced representation $\Ind\hat L$ used by Renault in
\cite{ren:jot91}.  Then we can invoke the deep results in
\cite{ren:jot91} to show, when $G$ is suitably amenable, that
\begin{equation*}
  K=\ker R=\ker(\Ind\hat L).
\end{equation*}
Since $\ker(\operatorname{ind}\tr)=\ker(\Ind\hat L)$ and since
$\ker(\operatorname{ind}\tr)$ is induced, this shows $K$ is induced
and completes the proof.

\section{The Proof of the Main Theorem}
\label{sec:proof-main-theorem}

In this section we present the details of the proof of
Theorem~\ref{thm:Main}. As in the statement of the theorem, $G$ will
always denote a second countable locally compact Hausdorff groupoid
endowed with a Haar system $\set {\lambda^{u}}_{u\in\go }$.

Let $K$ be a primitive ideal in $\Prim \cs(G)$.  Using Renault's
disintegration theorem, we can find an irreducible representation $R$
such that $\ker R=K$ \emph{and} such that $R$ is the integrated form
of a representation $(\mu,\go*\HH,V)$ of $G$.

We let $\so$ be the space of closed subgroups of $G$ equipped with the
Fell topology whose basic open sets are of the form
\begin{equation*}
  \mathcal{U}(K;U_{1},\dots,U_{n})=\set{H\in\so:\text{$H\cap
      K=\emptyset$ and $H\cap U_{i}\not=\emptyset$ for $i=1,2,\dots,n$}},
\end{equation*}
where $K\subset G$ is compact and each $U_{i}\subset G$ is open
(cf. \cite{wil:crossed}*{Appendix~H.1}).  Although $\so$ need not be
compact as in the group case, $\so\cup \set\emptyset$ is compact in
the space of closed subsets of $G$.  Hence, $\so$ is locally compact
Hausdorff.  Furthermore the map $p:\so\to\go$ given by $p(H)=u$ if
$r(H)=\set u=s(H)$ is continuous, and if $K\subset \go$ is compact,
then $p^{-1}(K)\cap \so$ is compact.  That is, $\so$ is conditionally
compact over $\go$ \cite{ren:jot91}*{\S1}.  We let $\Sigma$ be the
associated group bundle over $\so$:
\begin{equation*}
  \Sigma=\set{(u,H,\sigma):\text{$u=p(H)$ and $\gamma\in H$}}.
\end{equation*}
(The elements of $\Sigma$ have been written slightly redundantly to
make some of the subsequent computations easier to follow.)  Notice
that if $(u,H,\sigma)\in\Sigma$, then $H\subset G(u)$.  By
\cite{ren:jot91}*{Corollary~1.4}, there is a continuous Haar system
$\set{\beta^{H}}_{H\in\so}$ for $\Sigma$.

We want to define an action $\alpha$ of $G$ on $\cs(\Sigma)$ so that
$\bigl(\cs(\Sigma),G,\alpha\bigr)$ is a groupoid dynamical system.
We start by showing that $\cs(\Sigma)$ is a $C_{0}(\go)$-algebra
(cf. \cite{wil:crossed}*{Definition~C.1}).  For this, the following
variation on \cite{ren:jot91}*{Lemma~1.6} will be helpful.
\begin{lemma}
  \label{lem-omega}
  Let $G*\so=\set{(\sigma,H)\in G\times\so:s(\sigma)=p(H)}$.  Then
  there is a continuous map $\omega:G*\so\to(0,\infty)$ such that
  \begin{equation}\label{eq:8}
    \int_{H}f(\sigma\gamma\sigma^{-1})\,d\beta^{H}(\gamma) =
    \omega(\sigma,H) \int_{\sigma\cdot H}
    f(\gamma)\,d\beta^{\sigma\cdot H}(\gamma)\quad\text{for all $f\in
      C_{c}(G)$.} 
  \end{equation}
  Furthermore, for all $\sigma,\tau\in G$ and $H\in\so$, we have
  \begin{equation}\label{eq:9}
    \omega(\sigma\tau,H)=\omega(\tau,H)\omega(\sigma,\tau\cdot
    H)\quad\text{and} \quad
    \omega(\sigma,H)^{-1}=\omega(\sigma^{-1},\sigma\cdot H).
  \end{equation}
\end{lemma}
\begin{proof}
  [Sketch of the Proof] The existence of $\omega(\sigma,H)$ follows
  from the uniqueness of the Haar measure on $\sigma\cdot H$. The fact
  that $\omega$ is continuous follows from the fact that both
  integrals in \eqref{eq:8} are continuous with respect to
  $(\sigma,H)$.  Equation \eqref{eq:9} is a straightforward
  computation.
\end{proof}

For $u\in \go$, let $\Sigma_{G(u)}$ the compact Hausdorff space of
subgroups of $G(u)$.  Let $\cs(\Sigma_{G(u)})$ be the groupoid
\cs-algebra of the corresponding group bundle.  Thus if
$\Sigma_{G(u)}*G=\set{(H,\gamma)\in\so\times G:\gamma\in H}$, then
$\cs(\Sigma_{G(u)})$ is the completion of $C_{c}(\Sigma_{G(u)}*G) $ in
the obvious universal norm for the $*$-algebra structure given by
\begin{equation*}
  f*g(H,\gamma) 
  =\int_{H}f(H,\eta)g(H,\eta^{-1}\gamma)d\beta^{H}(\eta)
  \quad\text{and}\quad 
  f^{*}(H,\gamma)  =f(H,\gamma^{-1})^{*}.
\end{equation*}
Since the restriction map, $\kappa_{u}:C_{c}(\Sigma)\to
C_{c}(\Sigma_{G(u)})$ is surjective (by
\cite{wil:crossed}*{Lemma~8.54}), $\kappa_{u}$ extends to a
homomorphism of $\cs(\Sigma)$ onto $\cs(\Sigma_{G(u)})$.

\begin{remark}
  \label{rem-groupoid-vs-group}
  Notice that $\cs(\Sigma_{G(u)})$ is (isomorphic to) Fell's subgroup
  \cs-algebra as originally defined in \cite{fel:tams64} (or as a
  special case of \cite{wil:crossed}*{\S8.4}).  It is important to
  note that, since we are treating $\Sigma_{G(u)}*G$ as a groupoid,
  there are no modular functions in the formula above for the adjoint
  in contrast to the definitions in \cite{fel:tams64} or
  \cite{wil:crossed}.
\end{remark}

\begin{lemma}
  \label{lem-cox-alg}
  The groupoid $C^{*}$-algebra $C^{*}(\Sigma)$ is a
  $C_{0}(G^{(0)})$-algebra.  Moreover the fiber $C^{*}(\Sigma)(u)$
  over $u$ is isomorphic to $C^{*}(\Sigma_{G(u)})$.
\end{lemma}
\begin{proof}
  The groupoid $C^{*}$-algebra $C^{*}(\Sigma)$ is clearly a
  $C_{0}(\Sigma^{(0)})$-algebra, and as in the proof of
  \cite{wil:crossed}*{Proposition~8.55}, it is not hard to check that
  the fibre $\cs(\Sigma)(H)$ over $H$ is isomorphic to $\cs(H)$.  In
  particular, the restriction map $\iota_{H}:C_{c}(\Sigma)\to
  C_{c}(H)$ is surjective and extends to a homomorphism of
  $\cs(\Sigma)$ onto $\cs(H)$.
 
  By composing functions on $\go$ with $p$, we see that $\cs(\Sigma)$
  is also a $C_{0}(G^{(0)})$-algebra. Let $u\in G^{(0)}$ and let
  $I_{u}$ be the ideal of $C^{*}(\Sigma)$ spanned by
  $C_{0,u}(G^{(0)})\cdot C_{c}(\Sigma)$, where $C_{0,u}(G^{(0)})$
  consists of the functions $f$ in $C_{0}(G^{(0)})$ such that
  $f(u)=0$. Then $C^{*}(\Sigma)(u)=C^{*}(\Sigma)/I_{u}$.  Clearly
  $I_{u}\subset\ker\kappa_{u}$. To show that $C^{*}(\Sigma)(u)$ is
  isomorphic to $C^{*}(\Sigma_{G(u)})$ it is enough to prove that
  $I_{u} \supset\ker\kappa_{u}$.

  Let $L$ be a representation of $\cs(\Sigma)$ such that $I_{u}\subset
  \ker L$.  An approximation argument shows that if $f\in
  C_{c}(\Sigma)$ is such that $f(u,H,\gamma)=0$ for all $H\in G(u)$
  and $\gamma\in H$, then $f\in \ker L$.  Therefore if $\phi\in
  C_{0}(\Sigma)$ is such that $\phi(H)=1$ for all $H\in
  \Sigma_{G(u)}$, then $L(f)=L(\phi\cdot f)$.

  We can view $\Sigma_{G(u)}$ as a \emph{compact} subset of
  $\Sigma$. Since
  \begin{equation*}
    H\mapsto\int_{H}f(\gamma,H)\, d\beta^{H}(\gamma)
  \end{equation*}
  is continuous on $\Sigma$, for any $\epsilon>0$ we can find $\phi\in
  C_{0}(\Sigma)$ such that $\phi(H)=1$ for all $H\in\Sigma_{G(u)}$ and
  such that
  \begin{equation*}
    \|\phi\cdot
    f\|_{\cs(\Sigma)}\le\sup_{H\in\Sigma_{G(u)}}\|\iota_{H}(f)\|_{1}+\epsilon,
  \end{equation*}
  where $\iota_{H}:C_{c}(\Sigma)\to C_{c}(H)$ is the restriction
  map. It follows that
  \begin{equation*}
    \|L(f)\|\le
    \sup_{H\in\Sigma_{G(u)}}\|\iota_{H}(f)\|_{1}\le \|\kappa_{u}(f)\|_{I},
  \end{equation*}
  where $\|\cdot\|_{I}$ is the $I$-norm on
  $C_{c}(\Sigma_{G(u)}*G)\subset \cs(\Sigma_{G(u)})$.  Thus we can
  define a $\|\cdot\|_{I}$-decreasing representation $L'$ of
  $C_{c}(\Sigma_{G(u)})$ by $L'(\kappa_{u}(f)):=L(f)$. Since $L'$ must
  be norm decreasing for the \cs-norm, we have
  \begin{equation*}
    \|L(f)\|\le\|\kappa_{u}(f)\|_{\cs (\Sigma_{G(u)})},
  \end{equation*}
  and $\ker\kappa_{u}\subset \ker L$. Since $L$ is any representation
  with $I_{u}$ in its kernel, we have $\ker\kappa_{u}\subset I_{u}$.
\end{proof}

To define an action $\alpha$ of $G$ on $\cs(\Sigma)$ using
\cite{muhwil:nyjm08}*{Definition~4.1}, we first define
\begin{equation*}
  \alpha_{\eta}:\cs(\Sigma) \bigl(s(\eta)\bigr)\to
  \cs(\Sigma)\bigl(r(\eta)\bigr) 
\end{equation*}
at the level of functions by
\begin{equation}
  \label{eq:11}
  \alpha_{\eta}(F)\bigl(r(\gamma),H,\gamma\bigr):=
  \omega(\eta^{-1},H)^{-1} F\bigl(s(\eta),\eta^{-1}\cdot
  H,\eta^{-1}\gamma\eta\bigr). 
\end{equation}
Then we compute that
\begin{align}
  \int_{H}\alpha_{\eta}(F)\bigl(r(\eta),H&,\gamma\bigr)\,
  d\beta^{H}(\gamma) \nonumber \\
  &=\omega(\eta^{-1},H)^{-1}\int_{H}F\bigl(s(\eta),\eta^{-1}\cdot
  H,\eta^{-1}\gamma\eta\bigr)\, d\beta^{H}(\gamma)\label{eq:12}\\
  & =\int_{\eta^{-1}\cdot H}F\bigl(s(\eta),\eta^{-1}\cdot
  H,\gamma\bigr)\, d\beta^{\eta^{-1}\cdot H}(\gamma).\nonumber
\end{align}

\begin{lem}
  The triple $(C^{*}(\Sigma),G,\alpha)$ is a groupoid dynamical
  system.
\end{lem}
\begin{proof}
  The preceding discussion show that $\alpha_{\eta}$ is isometric for
  the $I$-norm, and hence defines an isomorphism. The fact that
  $\alpha_{\gamma\delta}=\alpha_{\gamma}\circ\alpha_{\delta}$ is clear
  by equation \eqref{eq:11}.  To see that
  $\alpha=\set{\alpha_{\eta}}_{\eta\in G}$ is continuous is a bit
  messy.  We'll use the criteria from \cite{muhwil:nyjm08}*{Lemma~4.3}
  and show that there is a $C_{0}(G)$-linear isomorphism
  \begin{equation*}
    \alpha:r^{*}\cs(\Sigma)\to s^{*}\cs(\Sigma)
  \end{equation*}
  which induces the $\alpha_{\eta}$ on the fibres.  It is not hard to
  establish that the pull-back $r^{*}\cs(\Sigma)$ is
  ($C_{0}(G)$-isomorphic to) the \cs-algebra of the group bundle
  $G*_{r}\Sigma^{(0)}*G=\set{(\eta,H,\gamma):\gamma\in H\subset
    G\bigl(r(\eta) \bigr)}$, and similarly for $s^{*}\cs(\Sigma)$.
  Then we can define
  \begin{equation*}
    \alpha:C_{c}(G*_{r}*\Sigma^{(0)}*G)\to C_{c}(G*_{s}*\Sigma^{(0)}*G)
  \end{equation*}
  by
  \begin{equation*}
    \alpha(f)(\eta,H,\gamma)=
    \omega(\eta^{-1},H)^{-1}f(\eta,\eta^{-1}\cdot 
    H,\eta^{-1}\gamma\eta) .
  \end{equation*}
  Then $\alpha$ is isometric with respect to the appropriate $I$-norms
  and therefore extends to a $C_{0}(G)$-linear isomorphism which
  induces the $\alpha_{\eta}$ as required.\end{proof}

\subsection{Restriction to the Stability Groups}
\label{sec:restr-stab-groups}

We maintain the set-up that $R$ is an irreducible representation with
$\ker R=K\in\Prim\cs(\Sigma)$, and that $R$ is the integrated form of
a representation $(\mu,\go*\HH, V)$ of $G$.  Note that
$\cs(\Sigma)=\Gamma_{0}(\go;\css)$ for an upper semicontinuous
\cs-bundle $p_{\css}:\css\to \go$ (as in
\cite{wil:crossed}*{Theorem~C.26}).  Since $u\mapsto G(u)$ is Borel
\cite{ren:jot91}*{Lemma~1.5}, we can define the \emph{restriction of
  $R$ to the isotropy groups of $G$} to be the representation $r$ of
$\cs(\Sigma)$ on $L^{2}(\go*\HH,\mu)$ given by
\begin{equation}\label{eq:13}
  r(F)h(u):=\int_{G(u)}
  F\bigl(u,G(u),\gamma\bigr)V_{\gamma}h(u)\Delta_{G(u)}(\gamma)^{-\half}
  \,d\beta^{G(u)}(\gamma),
\end{equation}
where $\Delta_{G(u)}$ is the modular function on $G(u)$ (see
Remark~\ref{rem-ru-groupoid} below).  It may be helpful to notice that
$r$ is the direct integral
\begin{equation}\label{eq:5}
  r=\dint_{\go}r_{u}\,d\mu(u),
\end{equation}
where $r_{u}$ is the composition of the representation of
$\cs\bigl(G(u)\bigr)$ given by $V\restr{G(u)}$ with the quotient map
$\kappa_{u}$ of $\cs(\Sigma)$ onto $\cs\bigl(G(u)\bigr)$.  (We will
also write $r_{u}$ for the representation of $\cs\bigl(G(u)\bigr)$.)

\begin{remark}
  \label{rem-ru-groupoid}
  Some care is necessary when applying $r_{u}$ to a function in $F\in
  C_{c}(\Sigma)$ --- or for that matter, $C_{c}\bigl(G(u)\bigr)$.
  Since $\Sigma$ is a groupoid, there is no modular function in the
  formula for the adjoint; instead, it must appear in the integrated
  form of representations in order that they be $*$-preserving. This
  ``explains'' the appearance of the group modular functions in
  \eqref{eq:13}.  In fact, $\deltau(u,\gamma)=\Delta_{G(u)}(\gamma)$
  is the Radon-Nikodym derivative of $\mu\circ\beta$ with respect to
  $\mu\circ\beta^{-1}$ associated to $\mu$ considered as a
  \emph{quasi-invariant} measure on $\go$ viewed as the unit space of
  the Borel groupoid group bundle $G'=\set{(u,\gamma):\gamma\in
    G(u)}$.  This observation will be important at the end of
  Section~\ref{sec:induc-repr}.
\end{remark}

\begin{lem}
  \label{lem-cov}
  The tuple $(r,\mu,\go*\HH, V )$ is a covariant representation of
  $\bigl(C^{*}(\Sigma),G,\alpha\bigr)$
  \cite{muhwil:nyjm08}*{Definition~7.9}. In fact,
  $V_{\gamma}r_{s(\gamma)}=(r_{r(\gamma)}\circ
  \alpha_{\gamma})V_{\gamma}$ for \emph{all} $\gamma\in G$.
\end{lem}
\begin{proof}
  We compute as follows. Fix $\gamma\in G$ with $s(\gamma)=v$ and
  $r(\gamma)=u$. Then
  \begin{align*}
    V_{\gamma}r_{v}(F) & =V_{\gamma}\int_{G(v)}F(v,G(v),\eta)V_{\eta}
    \Delta_{G(v)}(\eta)^{-\half}\,
    d\beta^{G(v)}(\eta)\\
    &
    =\int_{G(v)}F(v,G(v),\eta)V_{\gamma\eta}\Delta_{G(v)}(\eta)^{-\half}
    \Delta_{G(v)}(\eta)^{-\half}\, d\beta^{G(v)}(\eta)\\ &
    =\omega\bigl(\gamma,G(v)\bigr)\int_{\gamma\cdot
      G(v)}F\bigl(v,G(v),\gamma^{-1 }\eta\gamma\bigr)V_{\eta\gamma}
    \Delta_{G(v)}(\gamma^{-1}\eta\gamma) \,
    d\beta^{\gamma\cdot G(v)}(\eta)\\
    \intertext{which, since $\gamma\cdot G(v)=G(u)$ and
      $\Delta_{G(v)}(\gamma^{-1}\eta\gamma)=\Delta_{G(u)}(\eta)$, is}
    & =\int_{G(u)}\omega\bigl(\gamma,\gamma^{-1}\cdot
    G(u)\bigr)F\bigl(v,\gamma^{-1}\cdot
    G(u),\gamma^{-1}\eta\gamma\bigr)V_{\eta}
    \Delta_{G(u)}(\eta)^{-\half}\, d\beta^{G(u)}V_{\gamma}\\
    & =\int_{G(u)}\alpha_{\gamma}(F)
    \bigl(u,G(u),\eta\bigr)V_{\eta}\Delta_{G(u)}(\eta)^{-\half}\,
    d\beta^{G(u)}(\eta)V_{\gamma}\\
    & =r_{u}\bigl(\alpha_{\gamma}(F)\bigr)V_{\gamma}.
  \end{align*}
  The result follows.
\end{proof}

In view of Lemma~\ref{lem-cov}, we can let $\Lpp:=r\rtimes V$ be the
representation of the groupoid crossed product
$\cs(\Sigma)\rtimes_{\alpha}G$ which is the integrated form of
$(r,\mu,\go*\HH, V)$ (see \cite{muhwil:nyjm08}*{Proposition~7.11}).
If $\delta$ is the Radon Nikodym derivative of $\mu\circ \lambda$ with
respect to $\mu\circ\lambda^{-1}$, then for each $f\in
\Gamma_{c}(G;r^{*}\css)$
\begin{equation*}
  \bip(\Lpp(f)h|k)=
  \int_{\go}\int_{G}\bip(r_{u}\bigl(f(\gamma)\bigr)V_{\gamma} 
  h\bigl(s(\gamma)\bigr)|{k(u)})
  \delta(\gamma)^{-\frac12}\,d\lambda^{u}(\gamma) 
  \, d\mu(u).
\end{equation*}

Now we want to form Effros's ideal center decomposition of $r$
following \cite{ren:jot91}*{Theorem~2.2}.\footnote{Formally, Renault's
  proofs need to be modified to deal with upper semicontinuous
  \cs-bundles, but this is straightforward.}  Let
$\sigma:\Prim\cs(\Sigma)\to \go$ be the continuous map induced by the
$C_{0}(\go)$-structure on $\cs(\Sigma)$
\cite{wil:crossed}*{Proposition~C.5}.  As in the discussion preceding
\cite{ren:jot91}*{Proposition~1.14}, there is a continuous $G$-action
on $\primcss$, equipped with its Polish \emph{regularized} topology,
with respect to $\sigma$; that is, there is a continuous map
$(P,\gamma)\mapsto P\cdot\gamma$ from
\begin{equation*}
  \Prim \cs(\Sigma)*G=\set{(P,\gamma):\sigma(P)=r(\gamma)}
\end{equation*}
to $\primcss$.\footnote{Notice that in many treatments of groupoid
  actions on spaces, the structure map $\sigma$ is assumed to be open
  as well as continuous.  In this case, we have only that $\sigma$ is
  continuous.  Since $\Prim A$ has the regularized topology, there is
  no reason to suspect that $\sigma$ need be open even if $\css$ were
  a continuous \cs-bundle in the first place.}  Then, as in the
paragraph following the proof of \cite{ren:jot91}*{Proposition~1.14},
we can form the transformation groupoid
\begin{equation*}
  \G:=\Prim \cs(\Sigma)*G,
\end{equation*}
where
\begin{align*}
  r(P,\gamma)&=P&s(P,\gamma)&=P\cdot \gamma\\
  (P,\gamma)(P\cdot\gamma,\eta)&
  =(P,\gamma\eta)&(P,\gamma)^{-1}&=(P\cdot\gamma,\gamma^{-1}).
\end{align*}
With respect to the regularized topology on $\primcss$, $\G$ is what
Renault calls in \cite{ren:jot91} \emph{locally conditionally compact}
--- the important thing is that it is a standard Borel groupoid, and
that $\lambdabar^{P}=\epsilon_{P}\times\lambda^{\sigma(P)}$ is a
continuous Haar system for $\G$ \cite{ren:jot91}*{p.~12}.  Using this
structure, we want to construct an \emph{covariant ideal center
  decomposition} of $\Lpp$ in analogy with
\cite{wil:crossed}*{Appendix~G.2}.  The idea is to use a decomposition
theorem of Effros's \cite{eff:tams63} to decompose $r$ into
homogeneous representations in an equivariant way
(cf. \cite{wil:crossed}*{Theorem~C.22}).  Recall that a representation
$\pi$ of a \cs-algebra $A$ is homogeneous if $\ker\pi=\ker \pi^{E}$
for any nonzero projection $E\in\pi(A)'$. (Here $\pi^{E}$ is the
subrepresentation of $\pi$ corresponding to $E$.)  It was Sauvageot
who first noticed the importance of Effros's ideal center
decomposition for the solution of the Effros-Hahn conjecture.  A key
feature for us is that if $A$ is separable and $\pi$ is homogeneous,
then $\ker\pi$ is primitive \cite{wil:crossed}*{Corollary~G.9}.
Renault provides the decomposition result we need in
\cite{ren:jot91}*{Theorem~2.2}.  The essential features of his result
are as follows: $\Lpp$ is equivalent to a representation $\Lp$ on
$L^{2}(\Prim\cs(\Sigma)*\KK,\nu)$ where $\nu$ is a quasi-invariant
measure on $\G^{(0)}=\Prim\cs(\Sigma)$, and $\primcss*\KK$ is a Borel
Hilbert bundle over $\Prim\cs(\Sigma)$.  To define $\Lp$, Renault must
produce a $\nu$-conull set $U\subset \Prim\cs(\Sigma)$ and a Borel
homomorphism
\begin{equation*}
  \hat{\Lp}:\G\restr U\to\operatorname{Iso}\bigl(\Prim\cs(\Sigma)*\KK\bigr)
\end{equation*}
of the form
$\hat{\Lp}(P,\gamma)=\bigl(P,L_{(P,\gamma)},P\cdot\gamma\bigr)$, and
for each $P\in U$, homogeneous representations $\tr_{P}$ of
$\cs(\Sigma)$ with $\ker\tr_{P}=P$ such that for all $F\in\cs(\Sigma)$
\begin{align}
  &P\mapsto \tr_{P}(F)\quad\text{is Borel, and}\label{eq:2}\\
  &L(P,\gamma)\tr_{P\cdot\gamma}(F)
  =\tr_{P}(\alpha_{\gamma}(F))L(P,\gamma) \quad\text{for all
    $(P,\gamma)\in\G\restr{U}$}. \label{eq:3}
\end{align}
Since $\ker\tr_{P}= P$, we can always view $\tr_{P}$ as a
representation of the fibre $C^{*}(\Sigma)\bigl(\sigma(P)\bigr)$.
Therefore, if $f\in \Gamma_{c}(G;r^{*}\css)$, then $(\gamma,P)\mapsto
\tr_{P}\bigl(f(\gamma)\bigr)$ is well-defined and Borel on the set of
$(P,\gamma)$ such that $\sigma(P)=r(\gamma)$.  (Recall that sections
of the form $\gamma\mapsto \phi(\gamma)a\bigl(r(\gamma)\bigr)$ are
dense in $\Gamma_{c}(G;r^{*}\css)$ in the inductive limit topology.)
The primary conclusion of \cite{ren:jot91}*{Theorem~2.2} is that
$\Lpp$ is equivalent to the representation defined by
\begin{equation*}
  \Lp(f)h(P)=\int_{G}
  \tr_{P}\bigl(f(\gamma)\bigr)L(P,\gamma)\Delta(P,\gamma)^{-\frac12}
  h(P\cdot\gamma) \,d\lambda^{\sigma(P)}(\gamma),
\end{equation*}
where $\Delta$ is the Radon-Nikodym derivative of
$\nu\circ\lambdabar^{-1}$ with respect to $\nu\circ\lambdabar$.  Then
$\Lp$ is what we meant by a covariant ideal center decomposition of
$\Lpp$ (see Remark~\ref{rem-icd} below and compare with
\cite{wil:crossed}*{Proposition~G.24 and Lemma~9.9}).

Notice that \eqref{eq:3} implies that
\begin{equation}
  \label{eq:1}
  \tr_{P\cdot\gamma}\cong \gamma\cdot \tr_{P}\quad\text{for all
    $(P,\gamma)\in \G\restr U$.}
\end{equation}

On the other hand, in view of \eqref{eq:2}, we can form the direct
integral representation of $\cs(\Sigma)$ given by
\begin{equation}\label{eq:4}
  \tr:=\dint_{\Prim\cs(\Sigma)}\tr_{P}\,d\nu(P).
\end{equation}

\begin{lemma}
  \label{lem:r_and_tilde_r}
  Let $R$ be an irreducible representation of $\cs(G)$ and suppose
  that $r$ and $\tr$ are the representations of $\cs(\Sigma)$ defined
  in \eqref{eq:5} and \eqref{eq:4}, respectively.  Then $r$ and $\tr$
  are equivalent.
\end{lemma}

\begin{remark}\label{rem-icd}
  As a consequence of Lemma~\ref{lem:r_and_tilde_r}, we note that
  $\tr$ is an ideal center decomposition of $r$ as defined in
  \cite{wil:crossed}*{Definition~G.18}.  This justifies the
  terminology used above.
\end{remark}

\begin{proof}
  [Proof of Lemma~\ref{lem:r_and_tilde_r}] If $a\in
  \cs(\Sigma)=\Gamma_{0}(\go;\css)$ and $f\in\Gamma_{c}(G;r^{*}\css)$,
  then we can define $a\cdot f\in \Gamma_{c}(G;r^{*}\css)$ by
  \begin{equation*}
    a\cdot f(\gamma):=a\bigl(r(\gamma)\bigr)f(\gamma).
  \end{equation*}
  Then, viewing $\tr_{P}$ as a representation of the fibre
  $C^{*}(\Sigma)\bigl(\sigma(P)\bigr)$, we have $\tr_{P}\bigl(a\cdot
  f(\gamma)\bigr) =\tr_{P}(a)\tr_{P}\bigl(f(\gamma)\bigr)$.  Thus,
  \begin{align*}
    \Lp(a\cdot f)h(P)&=\int_{G} \tr_{P}\bigl(a\cdot
    f(\gamma)\bigr)L(P,\gamma)\Delta(P,\gamma)^{-\frac12}
    h(P\cdot\gamma) \,d\lambda^{\sigma(P)}(\gamma)\\
    &= \tr_{P}(a) \int_{G}
    \tr_{P}\bigl(f(\gamma)\bigr)L(P,\gamma)\Delta(P,\gamma)^{-\frac12}
    h(P\cdot\gamma) \,d\lambda^{\sigma(P)}(\gamma) \\
    &= \tr_{P}(a)\Lp(f)h(P).
  \end{align*}
  That is, $\Lp(a\cdot f)=\tr(a)\Lp(f)$.  Similarly, $\Lpp(a\cdot
  f)=r(a)\Lpp(f)$.

  Now let $M^{R}:L^{2}(\go*\HH,\mu)\to L^{2}(\Prim\cs(\Sigma),\nu)$ be
  a unitary isomorphism intertwining $\Lpp$ and $\Lp$.  Then we
  compute that for all $a\in C_{c}(\Sigma)$ and $f\in
  \Gamma_{c}(G;r^{*}\css)$ we have
  \begin{align*}
    M^{R}r(a)\Lpp(f)h&= M^{R}\Lpp(a\cdot f)h \\
    &= \Lp(a\cdot f)M^{R}h \\
    &= \tr(a)\Lp(f)M^{R}h \\
    &= \tr(a)M^{R}\Lpp(f)h.
  \end{align*}
  Since $\Lpp$ is nondegenerate, $M^{R}r(a)=\tr(a)M^{R}$ for all $a\in
  C_{c}(\Sigma)$.  Thus $r$ and $\tr$ are equivalent as claimed.
\end{proof}



A subset $U\subset\primcss$ is $\G$-invariant if $U\cdot\G\subset U$.
If $\nu$ is a quasi-invariant measure on $\primcss$ and if $V$ is
$\nu$-conull, then $\G\restr V$ is conull with respect to
$\nu\circ\lambdabar$.  We say that $U\subset\primcss$ is
\emph{$\nu$-essentially invariant} if there is a $\nu$-conull set $V$
such that $U\cdot\G\restr V\subset U$.  Notice that if $U$ is
$\nu$-essentially invariant, then $\phi=\charfcn U$ is
\emph{invariant} in the sense that $\phi\circ s=\phi\circ r$ for
$\nu\circ\lambdabar$-almost all $(P,\gamma)$.  In general, a
quasi-invariant measure $\nu$ is called \emph{ergodic} for the action
of $G$ on $\primcss$ if every Borel function $\phi$ on $\primcss$
which is invariant in the above sense is constant $\nu$-almost
everywhere (see \cite{ram:am71}*{p.~274}).  It is not hard to see that
it suffices to consider $\phi$ which are characteristic functions of a
Borel set.  Furthermore, it follows from the above and
\cite{ram:am71}*{Lemma~5.1} or \cite{muh:cbms}*{Lemma~4.9} that
$\phi=\charfcn U$ is invariant if and only if $U$ is $\nu$-essentially
invariant.\footnote{We thank the referee for pointing out the proper
  relationship between ``$\nu$-essentially invariant'' sets and the
  usual notion of ergodicity for measured groupoids as laid out in
  \cite{ram:am71}.}

\begin{prop}
  \label{prop:nu_ergodic}
  Let $G$ be a second countable Hausdorff groupoid $G$ with Haar
  system $\set {\lambda^{u}}_{u\in G^{(0)}}$. Assume that $R$ is an
  irreducible representation of $C^{*}(G)$. Let $r$ be the restriction
  of $R$ to the isotropy groups of $G$ defined by \eqref{eq:5}. Then
  the quasi-invariant measure $\nu$ on $\Prim\cs(\Sigma)$ in the ideal
  center decomposition \eqref{eq:4} of $r$ is ergodic with respect to
  the action of $G$ on $\Prim \cs (\Sigma)$.
\end{prop}

For the proof, it will be convenient to make the following observation
(compare with the first part of the proof of
\cite{muhwil:nyjm08}*{Theorem~7.12}).

\begin{lemma}
  \label{lem-marius}
  If $f\in \Gamma_{c}(G;r^{*}\css)$ and $\phi\in C_{c}(G)$, then
  define $\phi\cdot f\in \Gamma_{c}(G;r^{*}\css)$ by
  \begin{equation*}
    \phi\cdot
    f(\gamma)=\int_{G}\phi(\eta)
    \alpha_{\eta}\bigl(f(\eta^{-1}\gamma)\bigr) 
    \, d\lambda^{r(\gamma)}(\eta).
  \end{equation*}
  Then $\Lpp(\phi\cdot f)=R(\phi)\Lpp(f)$.
\end{lemma}
\begin{proof}
  We simply compute using Fubini's Theorem as follows:
  \begin{align*}
    \Lpp(\phi&{}\cdot f)h(u) = \int_{G} r_{u}\bigl(\phi\cdot
    f(\gamma)\bigr) V_{\gamma}h\bigl(r(\gamma)\bigr)
    \delta(\gamma)^{-\frac12}\,d\lambda^{u}(\gamma)  \\
    &= \int_{G}\int_{G} \phi(\eta)
    r_{u}\bigl(\alpha_{\eta}\bigl(f(\eta^{-1}\gamma) \bigr)\bigr)
    V_{\gamma} h\bigl(s(\gamma)\bigr) \delta(\gamma)^{-\frac12}
    \,d\lambda^{u}(\gamma) \,d\lambda^{u}(\eta) \\
    \intertext{which, after sending $\gamma\mapsto \eta\gamma$, is} &=
    \int_{G}\int_{G}\phi(\eta)r_{u}\bigl(\alpha_{\eta}\bigl(f(\gamma)
    \bigr)\bigr) V_{\eta\gamma}h\bigl(s(\gamma)\bigr)
    \delta(\eta\gamma)^{-\frac12} \,d\lambda^{s(\eta)}(\gamma)
    \,d\lambda^{u}(\eta) \\
    \intertext{which, in view of Lemma~\ref{lem-cov}, is} &=
    \int_{G}\phi(\eta)V_{\eta} \Bigl(
    \int_{G}r_{u}\bigl(f(\gamma)\bigr) V_{\gamma}
    h\bigl(s(\gamma)\bigr) \delta(\gamma)^{-\frac12} \,
    d\lambda^{s(\eta)}(\gamma) \Bigr) \delta(\eta)^{-\frac12} \,
    d\lambda^{u}(\eta) \\
    &= \int_{G}\phi(\eta)V_{\eta}\Lpp(f)h\bigl(s(\eta)\bigr)
    \delta(\eta)^{-\frac12} \,
    d\lambda^{u}(\eta) \\
    &= R(\phi)\Lpp(f)h(u).\qedhere
  \end{align*}
\end{proof}

\begin{proof}
  [Proof of Proposition~\ref{prop:nu_ergodic}] Recall that the
  diagonal operators are the multiplication operators $T_{\phi}$ for
  $\phi$ a bounded Borel function on $\primcss$ (see
  \cite{wil:crossed}*{Definition~F.13}).  Let
  $B\subset\Prim\cs(\Sigma)$ be a $\G\restr V$-invariant Borel subset
  for a $\nu$-conull set $V\subset \primcss$, and let
  $\phi=\charfcn{B}$.  Then $\phi$ is a bounded Borel function on
  $\Prim\cs(\Sigma)$ and we can let $E=T_{\phi}$ be the corresponding
  diagonal operator on $L^{2}(\Prim\cs(\Sigma)*\KK,\nu)$.  It will
  suffice to show that $E$ is either the identity or the zero
  operator.  Since for $\nu$-almost all $P$, $\phi(P\cdot
  \gamma)=\phi(P)$ for $\lambda^{\sigma(P)}$-almost all $\gamma$, it
  is clear from the definition of $\Lp$ that $E$ commutes with
  $\Lp(f)$ for all $f\in \Gamma_{c}(G;r^{*}\css)$.  Thus
  $E'':=(M^{R})^{-1}EM^{R}$ commutes with $\Lpp(f)$ for all $f\in
  \Gamma_{c}(G;r^{*}\css)$.  But if $\phi\in C_{c}(G)$, then using
  Lemma~\ref{lem-marius}, we have
  \begin{align*}
    R(\phi)E''\Lpp(f)h&= R(\phi)\Lpp(f)E''h\\
    &= \Lpp(\phi\cdot f)E''h\\
    &= E''R(\phi)\Lpp(f)h
  \end{align*}
  for all $\phi\in C_{c}(G)$, $f\in\Gamma_{c}(G;r^{*}\css)$ and $h\in
  L^{2}(\go*\HH,\mu)$.  Since $\Lpp$ is nondegenerate, $E''$ commutes
  with every $R(\phi)$.  Since $R$ is assumed irreducible, $E''$, and
  therefore $E$, must be trivial.
\end{proof}

Let $\cb(\go)$ and $\bb(\go)$ denote, respectively, the bounded
continuous and bounded Borel functions on $\go$.  If
$\phi\in\bb(\go)$, then we will write $T_{\phi}$ and
$T^{\Sigma}_{\phi\circ\sigma}$ for the corresponding diagonal
operators in $L^{2}(\go*\HH,\mu)$ and
$L^{2}(\Prim\cs(\Sigma)*\KK,\nu)$, respectively.  Notice that if
$\set{\phi_{i}}\subset\bb(\go)$ is a bounded sequence converging to
$\phi\in\bb(\go)$ $\mu$-almost everywhere, then by the dominated
convergence theorem, $T_{\phi_{i}}\to T_{\phi}$ in the strong operator
topology.

\begin{lemma}
  \label{lem-intertwine}
  The isomorphism $M^{R}$ which implements the equivalence between $r$
  and $\tr$ intertwines the diagonal operators on $L^{2}(\go*\HH,\mu)$
  and $L^{2}(\Prim\cs(\Sigma)*\KK,\nu)$.  In fact, we have
  \begin{equation}
    \label{eq:6}
    M^{R}T_{\phi}=T^{\Sigma}_{\phi\circ\sigma}M^{R}\quad\text{for all
      $\phi\in\bb(\go)$.} 
  \end{equation}
\end{lemma}
\begin{proof}
  If $\phi\in C_{0}(\go)$ and $F\in C_{c}(\Sigma)$, then
  \cite{wil:crossed}*{Proposition~C.5} (and the discussion preceding
  it) implies that
  \begin{equation*}
    \tr_{P}(\phi\cdot F)=\phi\bigl(\sigma(P)\bigr)\tr_{P}(F),
  \end{equation*}
  and it is not hard to see that this formula still holds when
  $\phi\in\cb(\go)$.\footnote{The action of $C_{0}(\go)$ is given by a
    \emph{nondegenerate} homomorphism of $C_{0}(\go)$ into the center
    of $M\bigl(\cs(\Sigma)\bigr)$ and so extends to
    $\cb(\go)=M\bigl(C_{0}(\go)\bigr)$.} Thus if $\phi\in\cb(\go)$ and
  $F\in C_{c}(\Sigma)$, then
  \begin{align*}
    M^{R}T_{\phi}r(F)&= M^{R}r(\phi\cdot F) \\
    &= \tr(\phi\cdot F)M^{R}\\
    &= T^{\Sigma}_{\phi\circ\sigma}\tr(F)M^{R} \\
    &= T^{\Sigma}_{\phi\circ\sigma}M^{R}r(F).
  \end{align*}
  Since $r$ is nondegenerate, we have shown that \eqref{eq:6} holds
  for all $\phi\in\cb(\go)$.

  Now suppose that $M$ is a $\mu$-null set.  We claim that
  $\sigma^{-1}(M)$ is $\nu$-null.  Since $\mu$ is a Radon measure and
  $\go$ is second countable, we may as well assume that $M$ is a
  $G_{\delta}$ subset of a compact set.  But then there is a bounded
  sequence $\set{\phi_{i}}\subset C_{0}^{+}(\go)$ such that
  $\phi_{i}\searrow \charfcn M$ \emph{everywhere}.  Then
  $\phi_{i}\circ \sigma\searrow \phi\circ\sigma$ everywhere.  It
  follows that in the strong operator topology, we have
  $T_{\phi_{i}}\to T_{\charfcn M}=0$ and $T^{\Sigma}_{\phi_{i}\circ
    \sigma}\to T^{\Sigma}_{\charfcn{\sigma^{-1}(M)}}$.  Since
  \eqref{eq:6} holds for continuous functions, it follows that
  $T^{\Sigma}_{\charfcn{\sigma^{-1}(M)}}=0$.  That is,
  $\sigma^{-1}(M)$ is $\nu$-null.

  Now if $\phi\in\bb(\go)$, then we can find a bounded sequence
  $\set{\phi_{i}}\subset \cb(\go)$ such that $\phi_{i}\to \phi$
  $\mu$-almost everywhere.  In view of the previous paragraph,
  $\phi_{i}\circ \sigma\to \phi\circ\sigma$ $\nu$-almost everywhere.
  Therefore $T_{\phi_{i}}\to T_{\phi}$ and
  $T^{\Sigma}_{\phi_{i}\circ\sigma}\to T^{\Sigma}_{\phi\circ \sigma}$
  in the strong operator topology, and since \eqref{eq:6} holds for
  each $\phi_{i}$, it follows that \eqref{eq:6} holds for all $\phi$.
\end{proof}

As we saw in the previous proof, $\sigma_{*}\nu\ll\mu$, where
$\sigma_{*}\nu$ is the push-forward of $\nu$ under $\sigma$:
$\sigma_{*}\nu(E)=\nu(\sigma^{-1}(E))$. Therefore, using
\cite{wil:crossed}*{Corollary~I.9}, we can disintegrate $\nu$ with
respect to $\mu$. This means that there are finite measures $\nu_{u}$
on $\primcss $ supported in $\sigma^{-1}(u)$ such that
\begin{equation*}
  \int_{\primcss
  }\phi(P)d\nu(P)=\int_{\go }\int_{\primcss
  }\phi(P)d\nu_{u}(P)d\mu(u)
\end{equation*}
for any bounded Borel function $\phi$ on $\primcss$.  Since
$P\mapsto\tr_{P}$ is a Borel field of representations, we can form the
direct integral representation
\begin{equation} \hr_{u}:=\int_{\primcss
  }^{\oplus}\tr_{P}d\nu_{u}(P)\label{eq:19}
\end{equation} on
$\mathcal{V} _{u}:=L^{2}(\primcss *\KK ,\nu_{u})$. We can form then the
Borel Hilbert bundle $\go *\VV $ induced from $L^{2}(\primcss
*\KK ,\nu)$ via the disintegration of $\nu$ with respect to $\mu$ (see
\cite{wil:crossed}*{Example~F.19}).  Since we can identify
$L^{2}(\primcss *\KK ,\nu)$ with $L^{2}(\go *\VV ,\mu)$, it
follows that $\tr$ is equivalent to
\begin{equation*}
  \hr=\int_{\go
  }^{\oplus}\hr_{u}d\mu(u).
\end{equation*}
If $\phi\in B^{b}(\go )$ and if $T_{\phi}^{\prime}$ is the
corresponding diagonal operator on $L^{2}(\go *\VV ,\mu)$, then
\begin{equation*}
  M^{R}T_{\phi}=T_{\phi}^{\prime}M^{R}.
\end{equation*}
Therefore \cite{wil:crossed}*{Corollary~F.34} implies that $r_{u}$ and
$\hr_{u}$ are equivalent for $\mu$-almost all~$u$.

Since $\ker\hr_{u}$ is separable, equation \eqref{eq:19} implies that
there exists a $\nu_{u}$-null set $N(u)$ such that
\begin{equation*}
  \ker\hr_{u}\subset\ker\tr_{P} \quad\text{if $P\notin N(u)$.}
\end{equation*}
Since $\supp \nu_{u}\subset\sigma^{-1}(u)$ we can rewrite this as
\begin{equation}
  \ker\hr_{\sigma(P)}\subset\ker\tr_{P}\label{eq:20}
\end{equation}
for $\nu_{u}$-almost all $P$ and for all $u$. It follows that
\eqref{eq:20} holds for $\nu$-almost all $P$. Thus off a $\nu$-null
set $N$, $\tr_{P}$ factors through
$\cs\bigl(G\bigl(\sigma(P)\bigr)\bigr)$.

\subsection{The Induced Representation}
\label{sec:induc-repr}

We are retaining the notation and assumptions from the previous
section: $R$ is an irreducible representation of $\cs(G)$ with kernel
$K$ and $r$ is the restriction of $R$ to the isotropy groups of $G$.
We have seen that if
\begin{equation*}
  \tr=\dint_{\primcss}\tr_{P}\,d\nu(P)
\end{equation*}
is the ideal center decomposition of $r$ defined in (\ref{eq:4}), then
$\tr_{P}$ factors through $\cs\bigl(G\bigl(\sigma(P)\bigr)\bigr)$ for
almost all $P$.  Next we want to form an induced representation
$\operatorname{ind}\tr$ of $\cs(G)$.  First we recall some of the
basics of induced representations of groupoids.

Suppose that $\rho$ is a representation of $\cs\bigl(G(u)\bigr)$.
Then, using the notation from \cite{ionwil:pams08}*{\S2}, we define
$\indgug\rho$ to be representation of $\cs(G)$ on the completion of
the $C_{c}(G_{u})\odot \H_{\rho}$ with respect to the inner product
defined on elementary tensors by
\begin{equation*}
  \ip(\phi\tensor h|\psi\tensor
  k)=\bip(\rho\bigl(\Rip<\psi,\phi>\bigr)h|k) .
\end{equation*}
Then
\begin{equation*}
  \indgug\rho(f)(\phi\tensor h)=f*\phi\tensor h.
\end{equation*}
If $\mathcal{I}(A)$ denotes the space of (closed two-sided) ideals in
$A$, then, as in \cite{rw:morita}*{Proposition~3.34 and
  Corollary~3.35}, there is a continuous map
\begin{equation*}
  \indgug:\mathcal{I}\bigl(\cs\bigl(G(u)\bigr)\bigr) \to
  \mathcal{I}\bigl(\cs(G)\bigr) 
\end{equation*}
characterized by
\begin{equation*}
  \ker\bigl(\indgug \rho\bigr)=\indgug(\ker\rho).
\end{equation*}
Since $\indgug\rho$ is irreducible if $\rho$ is
\cite{ionwil:pams08}*{Theorem~5}, it follows that $\indgug
J\in\Prim\cs(G)$ if $J\in\Prim\cs\bigl(G(u)\bigr)$.  Recall that we
call $K$ an \emph{induced primitive ideal} if $K$ is primitive and
$K=\indgug J$ for some $J\in \Prim\cs\bigl(G(u)\bigr)$.

The following is a straightforward consequence of the definitions.

\begin{lemma}
  \label{lem-ind-inv}
  Let $\rho$ be a representation of $\cs\bigl(G(u)\bigr)$ and let
  $\gamma\in G_{u}$.  Let $\gamma\cdot \rho$ be the representation of
  $\cs\bigl(G\bigl(r(\gamma)\bigr)\bigr)$ given by $\gamma\cdot
  \rho(a):= \rho(\gamma^{-1} a\gamma)$. Then $\indgug\rho$ and
  $\Ind_{G(r(\gamma))}^{G} \gamma\cdot \rho$ are equivalent
  representations.
\end{lemma}
\begin{proof}
  [Sketch of the Proof] Define $V:C_{c}(G_{u})\to
  C_{c}(G_{r(\gamma)})$ by
  \begin{equation*}
    V(f)(h)=\omega\bigl(\gamma,G(u)\bigr)^{\frac12} f(h\gamma).
  \end{equation*}
  Then
  \begin{align*}
    \rho\bigl(\Rip<\psi,\phi>\bigr)&= \int_{G(u)} \Rip<\psi,\phi>(h)
    \rho(h) \,d\beta^{G(u)}(h) \\
    &= \int_{G(u)}\int_{G} \overline{\psi(\eta)}\phi(\eta h) \rho(h)
    \,
    d\lambda_{u}(\eta) \,d\beta^{G(u)}(h) \\
    &= \int_{G(u)} \int_{G} \overline{\psi(\eta\gamma)}
    \phi(\eta\gamma h) \rho(h)
    \, \lambda_{r(\gamma)} (\eta) \,d\beta^{G(u)} (h) \\
    \intertext{which, in view of Lemma~\ref{lem-omega}, is} &=
    \omega\bigl(\gamma,G(u)\bigr)\int_{G(r(\gamma))} \int_{G}
    \overline{ \psi(\eta\gamma)} \phi(\eta h\gamma)\gamma\cdot\rho(h)
    \,
    d\lambda_{r(\gamma)}(\eta)\, d\beta^{G(r(\gamma))}(h) \\
    &= \gamma\cdot\rho\bigl(\Rip<V(\psi),V(\phi)>\bigr).
  \end{align*}
  Since $V$ is clearly onto, $f\tensor h\mapsto V(f)\tensor h$ extends
  to a unitary intertwining the two representations.
\end{proof}

Let $\tk(P)$ be the space of the induced representation
$\ind_{G(\sigma(P))}^{G}\tr_{P}$.  Thus, $\tk(P)$ is the completion of
of $C_{c}(G_{\sigma(P)})\odot \mathcal{K}(P)$ as described above.  Let
$\primcss*\tilde\KK=\set{(P,\tk(P):P\in\primcss}$ be the disjoint
union of the $\tk(P)$.  Since
\begin{equation*}
  P\mapsto \bip((\phi\tensor h)(P)|{\psi\tensor k)(P)})
\end{equation*}
is Borel, \cite{wil:crossed}*{Proposition~F.8} implies there is a
unique Borel structure on $\primcss*\tilde\KK$ making it into an
analytic Borel Hilbert bundle such that each $\phi\tensor h$ defines a
Borel section.  Since
\begin{equation*}
  \Ind_{G(\sigma(P))}^{G}\tr_{P}(\psi)\bigl((\phi\tensor h)(P)\bigr) =
  \psi*\phi\tensor h(P),
\end{equation*}
$P\mapsto \Ind_{G(\sigma(P))}^{G}\tr_{P}$ is a Borel field of
representations of $\cs(G)$.  Therefore, we can define the direct
integral representation
\begin{equation}
  \label{eq:7}
  \operatorname{ind}\tr
  :=\dint_{\primcss}\Ind_{G(\sigma(P)}^{G}\tr_{P}\,d\nu(P). 
\end{equation}

\begin{lemma}
  \label{lem-g-inv}
  Let $U\subset\primcss$ be the $\nu$-conull set associated to the
  equivariant ideal center decomposition $\tr$ of $r$.  Let
  $I_{P}:=\ker\bigl(\Ind_{G(\sigma(P)}^{G}\tr_{P}\bigr)$.  Then for
  all $P\in U$, $I_{P}\in\Prim\cs(G)$.  Furthermore, $P\mapsto I_{P}$
  is a Borel map of $U\subset \primcss$ into $\Prim\cs(G)$ such that
  $I_{P\cdot \gamma}=I_{P}$ for all $(P,\gamma)\in\G\restr U$.
\end{lemma}
\begin{proof}
  Since $\tr_{P}$ has kernel $P$, $I_{P}$ is primitive by the remarks
  preceding Lemma~\ref{lem-ind-inv}, and $P\mapsto I_{P}$ is Borel by
  \cite{wil:crossed}*{Lemma~F.28}.  Recall that by \eqref{eq:1}, there
  is a $\nu$-conull set $U$ such that $\tr_{P\cdot\gamma}$ is
  equivalent to $\gamma\cdot\tr_{P}$ if $(P,\gamma)\in\G\restr U$.
  Then for $(P,\gamma)\in\G\restr U$,
  \begin{align*}
    I_{P\cdot\gamma}&= \ker
    \bigl(\Ind_{G(\sigma(P\cdot\gamma))}^{G}\tr_{P\cdot \gamma}\bigr)
    \\
    &=\ker \bigl(\Ind_{G(s(\gamma))}^{G}\gamma^{-1}\cdot \tr_{P}\bigr) \\
    \intertext{which, by Lemma~\ref{lem-ind-inv}, is}
    &= \ker \bigl(\Ind_{G(\sigma(P))}^{G}\tr_{P}\bigr) \\
    &= I_{P}.\qedhere
  \end{align*}
\end{proof}

\begin{prop}
  \label{prop-prim-ind}
  Let $\ind\tr$ be the induced representation associated to an
  irreducible representation $R$ of $\cs(G)$ defined by \eqref{eq:7}.
  Then the kernel of $\ind\tr$ is an induced primitive ideal.
\end{prop}
\begin{proof}
  Let $\kappa:\primcss\to\Prim\cs(G)$ be a Borel map such that
  $\kappa(P):=I_{P}$ for $P\in U$.  If $B\subset \Prim\cs(G)$ is
  Borel, then $\kappa^{-1}(B)$ is $\nu$-essentially invariant.  Since
  $\nu$ is ergodic by Proposition~\ref{prop:nu_ergodic}, the proof of
  \cite{wil:crossed}*{Lemma~D.47} implies that $\kappa$ is essentially
  constant; that is, there is a $P_{0}$ such that
  $\ker\bigl(\Ind_{G(\sigma(P))}^{G}\tr_{P}\bigr)=I_{P_{0}}$ for
  $\nu$-almost all $P$.  But then $\ker(\ind\tr)=I_{P_{0}}$.  This is
  what we wanted.
\end{proof}

Let $r_{u}$ be as in (\ref{eq:5}) and let $\tH(u)$ be the space of the
induced representation $\indgug r_{u}$.  Thus $\tH(u)$ is the
completion of $C_{c}(G_{u})\atensor \H(u)$ with respect to the inner
product
\begin{equation*}
  \ip(\phi\tensor h|\psi\tensor
  k)=\bip(r_{u}\bigl(\Rip<\psi,\phi>\bigr)h|k) .
\end{equation*}
(There is no harm in taking $\phi$ and $\psi$ in $C_{c}(G)$ above.)
Let $\go*\tHH=\set{\bigl(u,\tH(u)\bigr):u\in\go}$ be the disjoint
union.  Then for each $\phi\tensor h\in C_{c}(G)\tensor
L^{2}(\go*\HH,\mu)$ we get a section of $\go*\tHH$ by
\begin{equation*}
  (f\tensor h)(u)=f\tensor h(u).
\end{equation*}
Then
\begin{align}
  \bip(\phi\tensor h(u)|{\psi\tensor k(u)}) &=
  \bip(r_{u}\bigl(\Rip<\psi,\phi>\bigr)h|k) \nonumber \\
  &= \int_{G(u)}\bip(\psi^{*}*\phi(\eta)V_{\eta}h(u)|{k(u)})
  \Delta_{G(u)}(\eta)^{-\half} \, d\beta^{G(u)}(\eta),\label{eq:15}
\end{align}
which is Borel in $u$.  Thus by
\cite{wil:crossed}*{Proposition~F.8}. there is a unique Borel
structure on $\go*\tHH$ making it into an analytic Borel Hilbert
bundle such that each $f\tensor h$ is a Borel section.  Since
\begin{equation*}
  \indgug r_{u}(\psi)\bigl(\phi\tensor h(u)\bigr) = \psi*\phi\tensor h,
\end{equation*}
it follows that $u\mapsto \indgug r_{u}$ is a Borel field of
representations and that we can make sense out of the direct integral
representation
\begin{equation*}
  \operatorname{ind}r:=\dint_{\go}\indgug r_{u}\,d\mu(u)
\end{equation*}
on $L^{2}(\go*\tHH,\mu)$.

\begin{proof}
  [Proof of Theorem~\ref{thm:Main}] Let $\Rind$ be the induced
  representation of $\cs(G)$ constructed by Renault on pages 16--17 of
  \cite{ren:jot91}.  After a bit of untangling and after specializing
  \cite{ren:jot91}*{Lemma~2.3} to our case, we see that there is a
  unitary $U$ mapping the space of $\Rind$ onto the completion of
  $C_{c}(G)\atensor L^{2}(\go*\HH,\mu)$ with respect to the inner
  product
  \begin{equation*}
    \bip(\phi\tensor h|\psi\tensor k)=\int_{\go}\bip(\phi\tensor h(u)|
    \psi\tensor k{(u)})\,d\mu(u),
  \end{equation*}
  where the integrand on the right-hand side is given by
  \eqref{eq:15}.  Moreover,
  \begin{equation*}
    (U^{*}\Rind(\psi)U)(\phi\tensor h)=\psi*\phi\tensor h.
  \end{equation*}
  Simply said: $\Rind$ is equivalent to $\operatorname{ind}r$.

  Let $M^{R}:L^{2}(\go*\HH,\mu)\to L^{2}(\primcss*\KK,\nu)$ be the
  unitary implementing the equivalence between $r$ and $\tr$, and then
  define $W:C_{c}(G)\atensor L^{2}(\go*\HH,\mu)\to C_{c}(G)\atensor
  L^{2}\bigl(\primcss*\KK,\nu)\bigr)$ by $W(\phi\tensor
  h):=\phi\tensor M^{R}h$. Since
  \begin{align*}
    W\circ \operatorname{ind}r(\psi)(\phi\tensor h) &=
    W(\psi*\phi\tensor M^{R}h) \\
    &= \operatorname{ind}\tr(\psi)(f\tensor M^{R}h) \\
    &= \operatorname{ind}\tr(\psi)\circ W(\phi\tensor h),
  \end{align*}
  it is not hard to see that $W$ extends to a unitary intertwining
  $\operatorname{ind}r$ and $\operatorname{ind}\tr$.  Therefore,
  $\operatorname {ind}\tr$ and $\Rind$ have the same kernel.  But then
  \cite{ren:jot91}*{Theorem~3.3} implies that $\ker R\subset
  \ker\bigl(\operatorname{ind}\tr\bigr)$.  On the other hand, if $G$
  is amenable as in the statement of the theorem, then
  \cite{ren:jot91}*{Theorem~3.6} implies that $K=\ker R=\ker
  \bigl(\operatorname{ind}\tr\bigr)$.  However, $\ker
  \bigl(\operatorname{ind}\tr\bigr)$ is an induced primitive ideal by
  Proposition~\ref{prop-prim-ind}.  This completes the proof.
\end{proof}

\def\noopsort#1{}\def\cprime{$'$} \def\sp{^}

\end{document}